\documentclass[journal]{IEEEtran}
\usepackage{amsmath,amsfonts}
\usepackage{amsthm}
\usepackage{array}
\usepackage[caption=false,font=footnotesize,labelfont=normalfont,textfont=normalfont]{subfig}
\usepackage{textcomp}
\usepackage{stfloats}
\usepackage{url}
\usepackage{verbatim}
\usepackage{graphicx}
\usepackage{cite}
\usepackage{cases}
\usepackage{xcolor}
\usepackage{pgfplots}
\pgfplotsset{compat=newest}
\usepackage{standalone}
\usepackage{multirow}

\usepackage[ruled,vlined,linesnumbered]{algorithm2e}
\SetArgSty{textnormal}

\newcommand{\m}[1]{\mathcal{#1}}
\renewcommand{\b}[1]{\boldsymbol{#1}}
\renewcommand{\t}[1]{\text{#1}}

\newtheorem{proposition}{Proposition}
\newtheorem{remark}{Remark}

\begin{document}

\title{Spatiotemporal Pricing and Fleet Management of Autonomous Mobility-on-Demand Networks: A Decomposition and Dynamic Programming Approach with Bounded Optimality Gap}

\author{\vspace{0.3cm} Zhijie Lai and Sen Li %
\thanks{This work was supported by the Hong Kong Research Grants Council under project 16202922, and the National Science Foundation of China under project 72201225. {(\em Corresponding author: Sen Li)}} %
\thanks{Z. Lai and S. Li are with the Department of Civil and Environmental Engineering, The Hong Kong University of Science and Technology, Clear Water Bay, Hong Kong (email: zlaiaa@connect.ust.hk, cesli@ust.hk).
}}

\maketitle

\begin{abstract}
This paper studies spatiotemporal pricing and fleet management for autonomous mobility-on-demand (AMoD) systems while taking elastic demand into account. We consider a platform that offers ride-hailing services using a fleet of autonomous vehicles and makes pricing, rebalancing, and fleet sizing decisions in response to demand fluctuations. A network flow model is developed to characterize the evolution of system states over space and time, which captures the vehicle-passenger matching process and demand elasticity with respect to price and waiting time. The platform's objective of maximizing profit is formulated as a constrained optimal control problem, which is highly nonconvex due to the nonlinear demand model and complex supply-demand interdependence. To address this challenge, an integrated decomposition and dynamic programming approach is proposed, where we first relax the problem through a change of variable, then separate the relaxed problem into a few small-scale subproblems via dual decomposition, and finally solve each subproblem using dynamic programming. Despite the nonconvexity, our approach establishes a theoretical upper bound to evaluate the solution optimality. The proposed model and methodology are validated in numerical studies for Manhattan. We find that compared to the benchmark case, the proposed upper bound is significantly tighter. We also find that compared to pricing alone, joint pricing and fleet rebalancing can only offer a minor profit improvement when demand can be accurately predicted. However, during unanticipated demand surges, joint pricing and rebalancing can lead to substantially improved profits, and the impacts of demand shocks, despite being more widespread, can dissipate faster.
\end{abstract}


\section{Introduction}
\IEEEPARstart{O}{wing} to technological advances in autonomous driving, the combination of on-demand ride-hailing services and autonomous vehicles (AVs), which we refer to as AMoD systems, presents an emerging paradigm that may reshape the landscape of future urban mobility. The AMoD system is deemed superior to its human-driven counterpart. One of the primary reasons lies in the fact that drivers' wage accounts for a significant proportion of operational expenditures in traditional ride-hailing platforms. In comparison, vehicle automation can circumvent the high labor cost, improve fleet efficiency, and accrue operational savings, which will fundamentally enhance the profitability of the ride-hailing industry. It is estimated that the cost of taxi services in San Francisco would drastically reduce by 74\% after automating the fleet \cite{becker_impact_2020}, and that replacing traditional taxis with shared and autonomous ones can downsize the fleet by 59\% without compromising the service level and thus lead to substantial reductions in carbon emissions \cite{lokhandwala_dynamic_2018}. 

Due to these advantages, operational strategies of AMoD platforms have attracted considerable attention from the research community. In the existing literature, many studies have focused on fleet management \cite{powell_fleet_2005}\cite{pillac_review_2013}, such as vehicle routing and relocation, to better service the time-varying demand over transportation networks. In fact,  fleet management of AMoD systems can be considered as the classic vehicle routing problem (VRP) \cite{dantzig_truck_1959}, which has been extensively studied in various forms, including dynamic VRPs \cite{pillac_review_2013},  VRP with Time Windows \cite{braysy_vehicle_2005}, pickup and delivery problems (PDPs) \cite{berbeglia_dynamic_2010}, and dynamic dial-a-ride problems (DARPs) \cite{cordeau_dial--ride_2007}. However, many existing works do not explicitly model passengers' demand elasticity with respect to price and waiting time. Both of these factors are crucial for travelers' decision-making in modern transportation systems, as riders have become increasingly sensitive to prices and service quality (e.g., waiting time) since the proliferation of various on-demand mobility services, such as Uber, Lyft, and Didi. Explicitly considering demand elasticity with respect to both price and waiting time is important in this context, but it poses significant modeling and operational challenges. The platform not only needs to determine the price over space and time to guide passenger demand but also needs to proactively relocate idle vehicles to steer the fleet distribution so that satisfactory service quality can be guaranteed. Among extensive research examining pricing \cite{sayarshad_scalable_2015}\cite{santos_taxi_2015} and relocation \cite{bent_waiting_2007}\cite{ma_dynamic_2019} strategies separately, few studies have jointly considered both while accounting for demand elasticity to price and waiting time. This is because the integration of these modeling elements gives rise to nonlinear dynamics and highly nonconvex optimal control problems. Consequently, global optimality cannot be guaranteed, and the optimality of the obtained solution is difficult to verify analytically. {\em To the best of our knowledge, we have not found any existing work that investigates joint spatiotemporal pricing and relocation of AMoD fleets with a comprehensive elastic demand model while at the same time characterizing the global optimality of the solution to the resulting nonconvex problem}.

\subsubsection*{Related Work}
Operations of AMoD systems are relevant to several mathematical problems, including dispatching, routing, scheduling, and rebalancing of service fleets. Examples include the VRP and its variants \cite{pillac_review_2013,dantzig_truck_1959,braysy_vehicle_2005}\cite{cordeau_dial--ride_2007}, where distance-minimized routes are to be planned to satisfy a set of transportation requests while considering different constraints. Within the context of AMoD systems, fleet operations have been examined primarily using queuing-theoretic models \cite{zhang_control_2016} \cite{iglesias_BCMP_2019}, simulation-based approaches \cite{fagnant_operations_2016}, and dynamic fluid models \cite{zhang_model_2016} \cite{salazar_congestion-aware_2019}. Particularly, Zhang et al. \cite{zhang_model_2016} addressed the optimal coordination of AMoD systems by devising a model predictive control algorithm, wherein the vehicle scheduling and routing problem was solved as a mixed integer linear program. Sayarshad and Chow \cite{sayarshad_non-myopic_2017} developed a queuing-based formulation for the relocation of MoD vehicles and solved the problem using a Lagrangian decomposition heuristic. Salazar et al. \cite{salazar_congestion-aware_2019} proposed a congestion-aware routing scheme for AMoD fleets and formulated a convex quadratic program based on a piecewise linear approximation for the Bureau of Public Roads (BPR) model. More recently, AMoD fleet operations have been extended to a decentralized manner. For instance, Chen et al. \cite{chen_decentralised_2021} proposed a decentralized cooperative cruising strategy for a fleet of autonomous taxis to maximize the total pickups during communication shutdowns. Duan et al. \cite{duan_centralized_2020} combined centralized and decentralized dispatching strategies to serve both immediate and reservation requests, where the decentralized dispatcher enables each vehicle to reserve capacity for long-term requests.

Despite a huge amount of research on fleet management and pricing strategies, only a handful of works have incorporated passengers’ sensitivity to price and waiting time. Some studies have utilized ride fare as a lever to manage demand \cite{sayarshad_scalable_2015}\cite{santos_taxi_2015}\cite{wollenstein-betech_joint_2020}, while waiting time was not considered. Given the nonlinear supply-demand interdependence, a complete and realistic demand model often yields nonconvex problems that are inherently challenging to solve exactly. As a result, existing studies that capture the aforementioned components primarily rely on approximate or learning-based algorithms for problem resolution. Among others, Lei et al. \cite{lei_path-based_2019} formulated the dynamic pricing and vehicle dispatching problem as a multi-period mathematical program with equilibrium constraints (MPEC) and solved the problem using approximate dynamic programming (ADP), where a piecewise linear value function approximation is introduced. Al-Kanj et al. \cite{al-kanj_approximate_2020} also utilized ADP to develop dispatching, repositioning, charging, and parking strategies, whereas the surge pricing policy was obtained separately using an adaptive learning approach embedded in the ADP framework. Turan et al. \cite{turan_dynamic_2020} used deep reinforcement learning to find a near-optimal pricing, routing, and charging policy for AMoD fleets. Nonetheless, these solution algorithms generally do not provide a proper metric (e.g., profit loss) to assess the derived operational decisions.

On the other hand, some studies focus on examining the optimality of different policies. For example, Braverman et al. \cite{braverman_empty-car_2019} established a fluid-based upper bound for optimally relocating empty cars to better meet passenger demand. Özkan and Ward \cite{ozkan_dynamic_2020} instead considered the matching problem in ride sharing systems and proved that their proposed policy achieves asymptotic optimality in a large market regime. However, neither of these studies captured pricing decisions and their impacts on passenger demand. Two other works considered model settings that are more similar to ours. Balseiro et al. \cite{balseiro_dynamic_2021} developed a dynamic pricing policy over a hub-and-spoke network and a performance bound based on Lagrangian relaxations that decompose the problem over spokes. Chen et al. \cite{chen_real-time_2023} approximated the stochastic dynamic pricing problem using a deterministic convex optimization, whose optimal solution serves as an upper bound for the revenue. Nevertheless, both models neglected the matching friction between vehicles and passengers, as well as the passenger's sensitivity to service quality. Our work fills the research gap by taking into account demand elasticity with respect to price and waiting time and meanwhile providing an upper bound.


\subsubsection*{Statement of Contributions}
The goal of this work is to establish a realistic mathematical model that incorporate demand elasticity with respect to both prices and waiting times to study joint pricing and fleet management of AMoD systems, while also establishing a theoretical upper bound to evaluate the global optimality of the solution. Specifically, our contributions are summarized as follows.
\begin{itemize}
    \item We formulate a network flow model to characterize the dynamics of AMoD systems\footnote{Throughout this paper, AMoD services specifically refer to on-demand ride-hailing services offered by AVs, but do not include other mobility services such as car-sharing, micro-transit, micro-mobility, ride-pooling, etc.}, capturing fundamental components such as vehicle-passenger matching, passenger waiting times, and demand elasticity with respect to both ride fare and waiting time. The comprehensive model encompasses a wide range of realistic scenarios and thus can facilitate understanding the operations of AMoD systems. The platform's profit maximization problem is formulated as a constrained optimal control problem, including prices, rebalancing flows, and fleet size adjustments as decision variables.

    \item We derive a theoretical upper bound to evaluate the solution optimality gap of the nonconvex optimization problem. To this end, we develop an integrated relaxation, decomposition, and dynamic programming approach, in which we first relax the original problem through a change of variable, then decouple the relaxed problem using dual decomposition, and finally employ dynamic programming to solve each small-scale subproblem. A tight surrogate upper bound is established, enabling us to measure the distance between the derived solution and the globally optimal solution despite nonconvexity. 

    \item We validate the proposed model with realistic case studies for Manhattan and present interesting managerial insights for AMoD operations. Our findings suggest that when demand prediction is accurate, joint pricing and fleet rebalancing do not lead to significant profit and throughput improvements compared to the case of pricing only. However, when demand prediction is seriously off, such as under an unexpected demand surge, the synergy between pricing and rebalancing becomes evident. Simulation results indicate that joint pricing and rebalancing decisions help dissipate the impacts of demand shocks faster despite their broad propagation. We further show that asymmetric demand patterns significantly affect the platform's operational strategies.
\end{itemize}

\section{Problem Formulation} \label{sec:formulation}
This section presents a network flow model to capture the system dynamics and details how AMoD operational strategies will affect the market outcomes.

\subsection{System Setup}
We consider a platform that provides ride-hailing services using an AV fleet in a city. The transportation network is represented as a graph $\m{G}(\m{K}, \m{E})$, where $\m{K} = \{1,\hdots,K\}$ is the set of zones, and $\m{E}$ is the set of edges. Each edge $(i,j) \in \m{E}$ connects a pair of adjacent zones. The trip time from zone $i$ to $j$, denoted as $\tau_{ij}$, is defined as the time traveling from zone $i$ to $j$ along the shortest path and is assumed to be exogenous (i.e., independent of the flow of ride-hailing vehicles). The total travel demand is exogenous and time-varying. Only a subset of travelers adopts AMoD services. Riders that request a trip from zone $i$ to $j$ are defined as type-$(i,j)$ passengers. 

Upon placing the ride request, each passenger will first wait for order confirmation (i.e., an idle car being assigned to service the passenger), then wait for being picked up by the assigned car, and at last, be delivered to her destination. Riders in these three statuses are defined as waiting passengers $Q^w_{i}(t)$, matched passengers $Q^m_{i}(t)$, and en-route passengers $Q^b_{ij}(t)$, respectively, where $t$ represents a time instant. On the supply side, according to whether vehicles are in operation or not, the entire fleet can be categorized into off-duty and on-duty vehicles. Off-duty vehicles, denoted as $N^p_{i}(t)$, are parked in each zone and unavailable to provide services. While on-duty vehicles are in one of the four statuses: (1) idling in each zone and waiting for the next passenger,\footnote{The difference between idle and off-duty vehicles is that idle vehicles are either cruising or queuing near the pickup point (e.g., taxi stand) and are therefore responsive to ride requests. Off-duty vehicles are parked and deactivated, thus cannot respond to further ride requests.} denoted as $N^v_{i}(t)$; (2) matched and on the way to pick up passengers, denoted as $N^m_{i}(t)$; (3) occupied and delivering passengers to destinations, denoted as $N^b_{ij}(t)$; and (4) relocating from zone $i$ to $j$ with empty seats, denoted as $N^r_{ij}(t)$. Let $N$ denote the total fleet size. The following vehicle conservation equation always holds:
\begin{multline}\label{eqn:vehicle conservation}
    N = \sum_{i\in\m{K}} N^v_{i}(t) + \sum_{i\in\m{K}} N^m_{i}(t) + \sum_{i\in\m{K}} \sum_{j\in\m{K}} N^b_{ij}(t) \\ + \sum_{i\in\m{K}} \sum_{j\in\m{K}} N^r_{ij}(t) + \sum_{i\in\m{K}} N^p_{i}(t).
\end{multline}
Note that ride-pooling services are not considered and thus en-route passengers and occupied vehicles have one-to-one correspondence, i.e., $Q^m_{i}(t) = N^m_{i}(t)$ and $Q^b_{ij}(t) = N^b_{ij}(t)$. As such, AMoD systems can be fully captured by the state vector $S(t) = [\b{Q}^w(t), \b{Q}^m(t), \b{Q}^b(t), \b{N}^v(t), \b{N}^r(t), \b{N}^p(t)]$, where $\b{Q}^w(t) = [Q^w_i(t)]_{i\in\m{K}}$, $\b{Q}^m(t) = [Q^m_i(t)]_{i\in\m{K}}$, $\b{Q}^b(t) = [Q^b_{ij}(t)]_{i,j\in\m{K}}$ correspond to passengers' state, and $\b{N}^v(t) = [N^v_i(t)]_{i\in\m{K}}$, $\b{N}^r(t) = [N^r_{ij}(t)]_{i,j\in\m{K}}$, and $\b{N}^p(t) = [N^p_i(t)]_{i\in\m{K}}$ represent the fleet state.

In response to the time-varying travel demand, the platform determines an operational strategy, including prices $\b{p}$, rebalancing flows $\b{r}$, and fleet size adjustments $\b{s}$, to maximize the profit, where $\b{p}$, $\b{r}$, and  $\b{s}$ represent vector trajectories over time. We emphasize that although the system state $S(t)$ is defined as a continuous-time vector that evolves continuously over time, the control decisions (such as $\b{p}$, $\b{r}$, and  $\b{s}$) will not be continuously updated by the platform in practice. Therefore, the control variables $(\b{p},\b{r},\b{s})$ are step-wise trajectories over time that are updated for every $T_c$ period of time (e.g., 5 min) and remain constant between every two steps. To specify $\b{p}$, we note that passenger ride fare depends on the trip origin\footnote{The origin-dependent pricing policy closely follows the industry practice. When demand goes up in some areas, Uber will decide the surge multipliers based on where the trip originates \cite{uber}. In this case, the ride fare only depends on the origin of the trip, which is consistent with our model.} and duration. Let $p_{i}(t)$ denote the charge rate for a trip starting from zone $i$ at time $t$. The ride fare for a type-$(i,j)$ trip is then $p_{i}(t) \tau_{ij}$. To specify $\b{r}$, we note that the platform can relocate idle and redundant vehicles from oversupplied zones to undersupplied ones, thus the number of vehicles being relocated from zone $i$ to support zone $j$ at time $t$ can be denoted as $r_{ij}(t)$. To specify $\b{s}$, we note that the platform adapts the fleet size to demand fluctuations by parking idle cars and/or activating parked cars, thus we denote $s_{i}(t)$ as the number of idle vehicles that will be parked in zone $i$, or off-duty vehicles parked in zone $i$ that will be activated. Specifically, $s_{i}(t) < 0$ represents idle vehicles go offline and get parked, whereas $s_{i}(t) > 0$ means off-duty vehicles parked in zone $i$ go online and become available to service passengers.

\subsection{Network Flow Model}
To logically unfold the system dynamics, we next present a network flow model based on the sequential process of each passenger’s arrival and departure. Specifically, passengers first make entry decisions following a demand model. These requests are then either confirmed by the platform or canceled by the passengers. Once the request is confirmed, the passenger is subsequently picked up, delivered, and dropped off.

\subsubsection*{Passenger Arrival} Let $\bar{q}_{ij}(t)$ be the total travel demand for type-$(i,j)$ trips and $q_{ij}(t)$ be the number of riders requesting services at time $t$. Passengers are sensitive to pickup time and ride fare. The demand is given by
\begin{equation} \label{eqn:demand function}
    q_{ij}(t) = \bar{q}_{ij}(t) e^{-\epsilon( \alpha w_{i}(t) + p_{i}(t) \tau_{ij} )},
\end{equation}
where $\epsilon$ is a sensitivity parameter, $\alpha$ is passengers' value of time, and $w_i(t)$ is the average waiting time for passengers to be picked up in zone $i$.\footnote{The ride confirmation time is typically significantly smaller than the pickup time (a few seconds v.s. a few minutes), thus for simplicity, we assume passenger demand is primarily elastic with respect to the pickup time $w_i(t)$. } The exponential term represents the proportion of travelers using AMoD services. Overall, this demand function is decreasing in ride fare and the quality of service, which characterizes the demand elasticity of passengers with respect to both prices and waiting times. 

\subsubsection*{Request Confirmation} New-coming passengers will enter the system and wait for request confirmation. The platform then confirms the order by matching each waiting passenger to an idle car. Empirical investigations in \cite{xu_generalized_2021} suggested that the matching between passengers and vehicles is $L$-shaped (see Fig. 4 in \cite{xu_generalized_2021}), i.e., waiting passengers will either be immediately matched under sufficient vehicle supply or queue up in the system for available cars. We thus specify the aggregate matching function as follows:
\begin{equation} \label{eqn:matching function}
    m_{ij}(t) = \dfrac{q_{ij}(t)}{\sum_{j\in\m{K}} q_{ij}(t)} \min \Big( Q^w_i(t), \max\big(N^v_i(t)-N^v_{\t{lb}},0\big) \Big),
\end{equation}
where $m_{ij}(t)$ represents the number of successful matches for type-$(i,j)$ trips at time $t$, and $N^v_{\t{lb}}$ is a threshold for the number of idle cars in each zone, which accounts for the practice that the platform will maintain the vehicle supply above some certain level to guarantee the quality of service. The interpretation of \eqref{eqn:matching function} is straightforward: the number of successful matches is the minimum between the number of waiting passengers and that of available idle vehicles. Note that the number of waiting passengers $Q^w_i(t)$ is tracked only by their origins but without destinations, whereas successful matches $m_{ij}(t)$ should include passengers' destination information. To characterize their relations, we use the destination distribution of new-coming passengers, i.e., the fractional term in \eqref{eqn:matching function}, to approximate that of waiting passengers. This is a reasonable approximation since travel demand distribution is usually stable for a short period and waiting passengers with different destinations are matched proportionally.

\subsubsection*{Request Cancellation} During busy hours, vehicles may be undersupplied, and waiting passengers accumulate in the system, leading to prolonged confirmation time. Passengers are typically impatient and will cancel the order when the confirmation time exceeds their maximum tolerance time. Intuitively, the request cancellation rate increases with the cumulative number of waiting passengers, but decreases with the number of idle cars. We adopt the following model to capture the request cancellation process:
\begin{equation} \label{eqn:cancallation}
    \tilde{m}_i(t) = \min \Big( Q^w_i(t), \max \big(0, c_0 + c_1 Q^w_i(t) + c_2 N^v_i(t) \big) \Big),
\end{equation}
where the number of request cancellations is linear to $Q^w_i(t)$ and $N^v_i(t)$ and truncated in $[0,Q^w_i(t)]$, and $c_0,c_1,c_2$ are weight parameters. Particularly, $c_1 > 0$ and $c_2 < 0$, which follows the intuition that the amount of request cancellations will increase if new-coming requests accumulate and congest the system, but decrease if idle vehicles are sufficiently supplied. These parameters depend on the distribution of passengers' maximum tolerance time and thus should be calibrated empirically. In this paper, we use microscopic simulations to validate this model and show in the Appendix that it can accurately capture passengers' request cancellation behavior.

\subsubsection*{Passenger Pickup} After order confirmation, passengers wait for the arrival of assigned cars. Assume that each passenger is matched to the nearest idle car. Therefore the pickup efficiency will improve with more available vehicles in the neighborhood. To characterize this dependence, the number of successful pickups is specified as
\begin{equation} \label{eqn:pickup function}
    b_{ij}(t) = \dfrac{q_{ij}(t)}{\sum_{j\in\m{K}} q_{ij}(t)} \beta_i Q^m_i (t) \left[N^v_i(t)\right]^{\theta_i}, 
\end{equation}
where $b_{ij}(t)$ depends on the number of matched passengers and idle vehicles, and $\beta_i$ and $\theta_i$ are zone-dependent parameters. The passenger pickup model \eqref{eqn:pickup function} is actually a special case of the Cobb-Douglas function, which is widely used to model the macroscopic vehicle-passenger meeting efficiency \cite{yang_equilibrium_2011,zha_economic_2016,nourinejad_ride-sourcing_2020} and has been validated in \cite{zhang_efficiency_2019} using real data.\footnote{Such a Cobb-Douglas functional form was once limited to equilibrium analysis and recently extended to time-varying conditions. See Fig. 3 in \cite{nourinejad_ride-sourcing_2020} for quantitative evaluation in non-equilibrium environments.} Similar to \eqref{eqn:matching function}, the destination distribution of new-coming passengers, i.e., the fractional term in \eqref{eqn:pickup function}, is used to approximate that of matched passengers $Q^m_i(t)$. In this case, passengers' pickup time in each zone can be further given by
\begin{equation} \label{eqn:pickup time}
    w_i(t) = \dfrac{Q^m_i (t)}{\sum_{j\in\m{K}} b_{ij}(t)} = \dfrac{1}{\beta_i} \left[N^v_i(t)\right]^{-\theta_i},
\end{equation}
which follows the Little's law and serves as an approximation in dynamical systems. This power-form passenger pickup time model has been extensively studied in the literature. See \cite{arnott_taxi_1996} and \cite{li_regulating_2019} for detailed derivations and see Fig. 2 in \cite{xu_generalized_2021} for empirical validation.

\subsubsection*{Trip Completion} Passengers exit from the system after being dropped off. The trip completion rate depends on the cumulative number of in-vehicle passengers, trip duration, and congestion levels \cite{ramezani_dynamic_2018}. Let $\tilde{q}_{ij}(t)$ denote the number of type-$(i,j)$ passengers being dropped off at time $t$. Clearly, the trip completion model should satisfy
\[\dfrac{\partial \tilde{q}_{ij}(t)}{\partial Q^b_{ij}(t)} > 0, \dfrac{\partial \tilde{q}_{ij}(t)}{\partial\tau_{ij}} < 0, \text{and} \lim_{Q^b_{ij}(t) \rightarrow 0} \tilde{q}_{ij}(t) = 0.\]
\noindent The monotonicity property follows our intuition that more passengers will get off the car if more vehicles are currently on the way to destinations, while fewer vehicles will arrive if the trip lasts a longer time. In this paper, we use the linear form in \cite{nourinejad_ride-sourcing_2020} to characterize the trip completion process:
\begin{equation} \label{eqn:trip completion}
    \tilde{q}_{ij}(t) = \dfrac{\kappa}{\tau_{ij}} Q^b_{ij}(t),
\end{equation}
where $\kappa > 0$ is a scaling parameter depending on traffic conditions and other factors. We evaluate \eqref{eqn:trip completion} with empirical data and find that $\tilde{q}_{ij}(t)$ and $Q^b_{ij}(t)$ demonstrate a significant linear dependent relationship (see Appendix). This model also applies to rebalancing trips. Let $\tilde{r}_{ij}(t)$ denote the number of type-$(i,j)$ relocating vehicles that complete rebalancing trips at time $t$. Consistent with \eqref{eqn:trip completion}, $\tilde{r}_{ij}(t)$ is given by
\begin{equation} \label{eqn:reposition completion}
    \tilde{r}_{ij}(t) = \dfrac{\kappa}{\tau_{ij}} N^r_{ij}(t).
\end{equation}

\subsubsection*{System Dynamics} Based on the above discussion, we can summarize the system dynamics in the following equations:
\begin{subequations} \label{eqn:dynamics}
\begin{align}
    & \dfrac{\mathrm{d} Q^w_{i}(t)}{\mathrm{d} t} = \sum_{j\in\m{K}} q_{ij}(t) - \sum_{j\in\m{K}} m_{ij}(t) - \tilde{m}_{i}(t), \label{eqn:dynamics Qw}\\
    & \dfrac{\mathrm{d} Q^m_{i}(t)}{\mathrm{d} t} = \sum_{j\in\m{K}} m_{ij}(t) - \sum_{j\in\m{K}} b_{ij}(t), \label{eqn:dynamics Qm}\\
    & \dfrac{\mathrm{d} Q^b_{ij}(t)}{\mathrm{d} t} = b_{ij}(t) - \tilde{q}_{ij}(t), \vphantom{\sum_{j\in\m{K}}} \label{eqn:dynamics Qb}\\
    & \dfrac{\mathrm{d} N^v_{i}(t)}{\mathrm{d} t} = s_{i}(t) + \sum_{j\in\m{K}} \tilde{q}_{ji}(t) + \sum_{j\in\m{K}} \tilde{r}_{ji}(t) \label{eqn:dynamics Nv} \\
    & \phantom{\dfrac{\mathrm{d} N^v_{i}(t)}{\mathrm{d} t} = \quad \quad \quad \quad} - \sum_{j\in\m{K}} m_{ij}(t) - \sum_{j\in\m{K}} r_{ij}(t), \notag \\
    & \dfrac{\mathrm{d} N^r_{ij}(t)}{\mathrm{d} t} = r_{ij}(t) - \tilde{r}_{ij}(t), \vphantom{\sum_{j\in\m{K}}} \label{eqn:dynamics Nr}\\
    & \dfrac{\mathrm{d} N^p_{i}(t)}{\mathrm{d} t} = -s_{i}(t), \vphantom{\sum_{j\in\m{K}}} \label{eqn:dynamics Np}     
\end{align}
\end{subequations}
which reveal the evolution of system states and capture the intricate supply-demand interaction spatially and temporally. Such a macroscopic network flow model can be calibrated to emulate the behaviors of a realistic microscopic system, making it a powerful tool for investigating how distinct operational strategies will affect the system performance.

\section{Derivation of Solutions with Bounded Optimality Gap} \label{sec:solution framework}
In this section, we formulate the profit-maximization problem for AMoD platforms, followed by an integrated relaxation, decomposition, and dynamic programming approach, which is developed to establish a theoretical upper bound for evaluating the global optimality of the derived solution.

\subsection{Profit Maximization}
The platform aims to maximize its total profit over the operation horizon $\m{T}$ by finding an optimal operational strategy. The pricing, rebalancing, and fleet sizing decisions are adjusted after every $T_c$ period of time to accommodate the time-varying demand, and remain constant between every two adjustments, while system states are updated continuously over time and affected by the platform's decisions through the dynamic model \eqref{eqn:dynamics}. The platform's profit maximization problem can be cast as:
\begin{subequations} \label{eqn:problem}
\begin{align}
    \max_{\b{p},\b{r},\b{s}} \ & \int\limits_{t\in\m{T}} \Bigg [ \sum_{i\in\m{K}} \sum_{j\in\m{K}} m_{ij}(t) p_{i}(t) \tau_{ij} - \gamma N^\t{on}(t) \Bigg] \mathrm{d}t \label{eqn:problem objective} \\
    s.t. \quad & \t{system dynamics \eqref{eqn:dynamics}},  \notag \\
    & r_{ij}(t) \geq 0, \ \forall i,j,t, \label{eqn:problem r bound} \\
    & 0 \leq p_{i}(t) \leq p_\t{ub}, \ \forall i,t, \label{eqn:problem p bound} \\
    & N^v_{i}(t) \geq N^v_{\t{lb}}, \ \forall i,t, \label{eqn:problem Nv bound}\\
    & 0 \leq N^p_{i}(t) \leq N^p_{\t{ub},i}(t), \ \forall i,t, \label{eqn:problem Np bound}\\
    & Q^w_{i}(t), Q^m_{i}(t), Q^b_{ij}(t), N^b_{ij}(t) \geq 0, \ \forall i,j,t, \label{eqn:problem state bound}
\end{align}
\end{subequations}
where $N^\t{on}(t) = N - \sum_{i\in\m{K}} N^p_i(t)$ represents the total number of vehicles in operations. The first term in \eqref{eqn:problem objective} represents the total revenue collected from passengers. Assume that each on-duty vehicles incurs an operational cost $\gamma$ per unit of time, while those off-duty ones will not bring about extra costs.\footnote{In practice, off-duty vehicles also lead to costs, e.g., parking fees. Considering the average cost of off-duty vehicles is typically lower than that of on-duty ones, we can normalize the cost to zero for simplicity.} Thereby the second term in \eqref{eqn:problem objective} represents operational costs for the entire fleet. Constraints \eqref{eqn:problem r bound}-\eqref{eqn:problem state bound} specify the bounds for decision variables and system states. Particularly, an upper bound $p_\t{ub}$ is imposed to prevent the unrealistically high price. In \eqref{eqn:problem Np bound}, $N^p_{\t{ub},i}(t)$ is an exogenous variable that accounts for the availability of parking spaces. Besides, operation decisions are also implicitly subject to physical constraints. For example, if the fleet size adjustment $s_i(t)$ is too large, it will lead to a negative number of off-duty cars in some certain zone, which violates \eqref{eqn:problem Np bound}. On the other hand, if $s_i(t)$ is too small (negative), it will lead to a negative number of idle vehicles, which violates \eqref{eqn:problem Nv bound}. Similarly, the rebalancing decisions are also implicitly bounded by \eqref{eqn:problem Nv bound}. Overall, \eqref{eqn:problem} is a constrained and nonconvex optimization problem. We emphasize that although the dynamics are formulated in continuous time, the optimization problem (\ref{eqn:problem}) is essentially equivalent to a discrete-time control problem because the control decisions are adjusted only for every $T_c$ period of time, and each control action will uniquely determine the system state at the time instant when the next decision is made.

\subsection{Derivation of Upper Bounds}
\label{upperbound_sec}
The complex demand and supply models lead to a nonlinear and nonconvex objective function and also introduce strong coupling among distinct decision variables, making the profit-maximization problem \eqref{eqn:problem} difficult to solve. For this reason, none of the existing works can characterize the global optimality of solutions while capturing these relations. In the rest of this subsection, we propose a novel decomposition and dynamic programming approach that can evaluate the solution optimality by offering a theoretical upper bound.

\subsubsection*{Model Relaxation}
It is clear that most complexities of problem \eqref{eqn:problem} arise from the high-dimensional state space, which is unscalable due to the existence of $Q^b_{ij}(t)$ and $N^r_{ij}(t)$. Moreover, the interdependence between states of different zones also exacerbates intractability. For instance, trip completion models in \eqref{eqn:dynamics Nv}, i.e., $\tilde{q}_{ij}(t)$ and $\tilde{r}_{ij}(t)$, relate the state of zone $i$ to that of other zones. Thereby, the states of different zones are coupled with each together, leading to a complicated decision process over the network. In this light, we relax the dynamical system and transform it into a decomposable one such that the complex and high-dimensional problem can be separated into smaller-scale problems for each zone.

First, we split $Q^b_{ij}(t)$ into intra- and inter-zone passengers based on their destinations, i.e.,
\begin{equation}
    \sum_{j\in\m{K}} Q^b_{ij}(t) = Q^{\t{intra}}_{i}(t) + Q^{\t{inter}}_{i}(t),
\end{equation}
where $Q^{\t{intra}}_{i}(t) = Q^b_{ii}(t)$ denotes the number of in-vehicle passengers en route to their destinations inside zone $i$, and $Q^{\t{inter}}_{i}(t) = \sum_{j\in\m{K} \setminus i} Q^b_{ij}(t)$ denotes the total number of in-vehicle passengers on the way to their destinations outside zone $i$. Similarly, completed trips $\tilde{q}_{ij}(t)$ are also categorized into two classes, i.e., $\tilde{q}^{\t{intra}}_{i}(t)$ and $\tilde{q}^{\t{inter}}_{i}(t)$, which follow the linear trip completion model \eqref{eqn:trip completion} and are given by
\begin{subequations} \label{eqn:relaxed trip completion}
\begin{align}
    \tilde{q}^{\t{intra}}_{i}(t) &= \dfrac{\kappa}{\tau_{ii}} Q^{\t{intra}}_{i}(t), \\
    \tilde{q}^{\t{inter}}_{i}(t) &= \dfrac{\kappa}{\bar{\tau}_i} Q^{\t{inter}}_{i}(t),
\end{align}
\end{subequations}
where $\bar{\tau}_i = \min \{\tau_{ij}\}_{j \in \m{V} \setminus i}$ represents the minimum travel time for an inter-zone trip departing from zone $i$. This relaxation implies an assumption that all outbound trips from zone $i$ share a homogeneous travel time $\bar{\tau}_i$, which does not have any physical interpretation but is only to disentangle the state interdependence associated with $\tilde{q}_{ij}(t)$ so as to reduce the dimensionality of state space.

Second, we consider the number of idle vehicles in each zone, i.e., $N^v_{i}(t)$, as a decision variable in place of the rebalancing flow $r_{ij}(t)$. An intuitive interpretation for this relaxation is that the platform can arbitrarily allocate all the {\em unoccupied} vehicles in the system, including those idle and off-duty, to any desired locations. This enables the platform to directly control the number of on-duty cars by manipulating $N^v_{i}(t)$. Thus $s_{i}(t)$ can be removed from the action space and $N^p_{i}(t)$ can also be removed from the state space. Correspondingly, the second inequality constraint in \eqref{eqn:problem Np bound} is dropped since the spatial distribution of parked vehicles is no longer tracked as system states. We emphasize that the change of decision variables does not neglect the rebalancing flows. Instead, it actually implies that all relocation trips will be completed instantaneously, i.e., $N^r_{ij}(t) = 0, \forall i,j,t$. Therefore, this relaxation not only reduces the dimensionality of state and action space, but also decouples the state interdependence associated with $\tilde{r}_{ij}(t)$. {\em Note that $N^v_{i}(t)$ needs to be determined every single time step, which is different from the original model.} This is because the aggregation of inter-zone trips compromises the destination information of en-route passengers, preventing us to precisely track the evolution of idle vehicles. Hence, $N^v_{i}(t)$ is also removed from the state space. 

After these two relaxations, the platform's decisions become $p_i(t)$ and $N^v_i(t)$. System states become $Q^w_{i}(t)$, $Q^m_{i}(t)$, $Q^{\t{intra}}_{i}(t)$, and $Q^{\t{inter}}_{i}(t)$, yielding the simplified dynamics:
\begin{subequations} \label{eqn:relaxed dynamics}
\begin{align}
    & \dfrac{\mathrm{d} Q^w_{i}(t)}{\mathrm{d} t} = \sum_{j\in\m{K}} q_{ij}(t) - \sum_{j\in\m{K}} m_{ij}(t) - \tilde{m}_{i}(t), \label{eqn:relaxed dynamics Qw}\\
    & \dfrac{\mathrm{d} Q^m_{i}(t)}{\mathrm{d} t} = \sum_{j\in\m{K}} m_{ij}(t) - \sum_{j\in\m{K}} b_{ij}(t), \label{eqn:relaxed dynamics Qm}\\
    & \dfrac{\mathrm{d} Q^{\t{intra}}_{i}(t)}{\mathrm{d} t} = b_{ii}(t) - \tilde{q}^{\t{intra}}_{i}(t), \label{eqn:relaxed dynamics Qintra}\vphantom{\sum_{j\in\m{K}}}\\
    & \dfrac{\mathrm{d} Q^{\t{inter}}_{i}(t)}{\mathrm{d} t} = \sum_{j\in\m{K} \setminus i} b_{ij}(t) - \tilde{q}^{\t{inter}}_{i}(t). \label{eqn:relaxed dynamics Qinter}
\end{align}
\end{subequations}
Further, the original optimization problem is relaxed to
\begin{subequations} \label{eqn:relaxed problem}
\begin{align}
    \max_{\b{p},\b{N}^v} \quad & \int\limits_{t\in\m{T}} \Bigg [ \sum_{i\in\m{K}} \sum_{j\in\m{K}} m_{ij}(t) p_{i}(t) \tau_{ij} - \gamma N^\t{on}(t) \Bigg] \mathrm{d}t \label{eqn:relaxed problem objective} \\
    s.t. \quad & \t{relaxed system dynamics \eqref{eqn:relaxed dynamics}},  \notag \\
    & 0 \leq p_{i}(t) \leq p_\t{ub}, \ \forall i,t, \\
    & N^v_{i}(t) \geq N^v_{\t{lb}}, \ \forall i,t, \\
    & Q^w_{i}(t), Q^m_{i}(t), Q^{\t{intra}}_{i}(t), Q^{\t{inter}}_{i}(t) \geq 0, \ \forall i,t, \\
    & N - N^\t{on}(t) \geq 0, \ \forall t, \label{eqn:relaxed problem constraint}
\end{align}
\end{subequations}
where $N^\t{on}(t) = \sum_{i\in\m{K}}[N^v_i(t) + Q^m_{i}(t) + Q^{\t{intra}}_{i}(t) + Q^{\t{inter}}_{i}(t)]$, and \eqref{eqn:relaxed problem constraint} corresponds to the fleet size constraint. Compared with the original problem, \eqref{eqn:relaxed problem} neglects some physical constraints in the system dynamics (e.g., travel time for repositioning trips), enables the platform to operate the fleet with more freedom, and thus promises an improved profit. This property is formalized as follows.
\begin{proposition} \label{prop:upper bound}
    The globally optimal solution to the relaxed problem \eqref{eqn:relaxed problem} is an upper bound for the original problem \eqref{eqn:problem}.
\end{proposition}
\begin{proof}
    Let $\{\b{p}^*,\b{r}^*,\b{s}^*\}$ be the globally optimal solution to \eqref{eqn:problem} and $\b{S}^*$ be the consequent trajectory of system states.
    Based on \eqref{eqn:problem Np bound}, we have
    \begin{multline}\label{eqn:original nonnegativity}
        N - \sum_{i\in\m{K}} N^{v*}_{i}(t) - \sum_{i\in\m{K}} Q^{m*}_{i}(t) \\ - \sum_{i\in\m{K}} \sum_{j\in\m{K}} Q^{b*}_{ij}(t) - \sum_{i\in\m{K}} \sum_{j\in\m{K}} N^{r*}_{ij}(t) \geq 0.
    \end{multline}
    We prove the proposition by showing that $\{\b{p}^*,\b{N}^{v*}\}$ is always feasible for \eqref{eqn:relaxed problem}. Denote by $\hat{S}$ the resultant state trajectory after implementing $\{\b{p}^*,\b{N}^{v*}\}$ in the relaxed system \eqref{eqn:relaxed dynamics}. Besides, the evolution of $Q^w$ and $Q^m$ remains unchanged after the relaxation, i.e., $Q^{w*}_{i}(t) = \hat{Q}^w_{i}(t)$ and $Q^{m*}_{i}(t) = \hat{Q}^m_{i}(t)$. Therefore the realized passenger demands along both trajectories are the same, i.e., $q^*_{ij}(t) = \hat{q}_{ij}(t)$. Summing up \eqref{eqn:relaxed dynamics Qintra} and \eqref{eqn:relaxed dynamics Qinter} and comparing the combination with \eqref{eqn:dynamics Qb}, we have
    \begin{equation} \label{eqn:compare Qb}
        \sum_{j \in \m{K}} Q^{b*}_{ij}(t) \geq \hat{Q}^{\t{intra}}_{i}(t) + \hat{Q}^{\t{inter}}_{i}(t), \; \forall i,t.
    \end{equation}
    Note that $\hat{N}^{r}_{ij}(t) = 0$ after the relaxation. Inserting \eqref{eqn:compare Qb} into \eqref{eqn:original nonnegativity} gives rise to
    \begin{equation} \label{eqn:relaxed nonnegativity}
        N - \sum_{i\in\m{K}} \left[ N^{v*}_{i}(t) + \hat{Q}^m_{i}(t) + \hat{Q}^{\t{intra}}_{i}(t) + \hat{Q}^{\t{inter}}_{i}(t) \right] \geq 0,
    \end{equation}
    which satisfies \eqref{eqn:relaxed problem constraint}, guarantees the feasibility of $\{\b{p}^*,\b{N}^{v*}\}$ in \eqref{eqn:relaxed problem}, and thus completes the proof.
\end{proof}
The above discussion indicates that the solution to \eqref{eqn:relaxed problem} is always a surrogate upper bound for \eqref{eqn:problem} such that the solution optimality is verifiable despite its nonconvexity. Next, We will show that the relaxed problem has special structures that can be leveraged in developing the upper bound.

\subsubsection*{Decomposition and Dynamic Programming}
Although the proposed relaxation scheme considerably simplifies the model, the relaxed problem remains nonconvex. Fortunately, both the objective function and system dynamics in \eqref{eqn:relaxed problem} are separable. Only constraint \eqref{eqn:relaxed problem constraint} is coupled, enabling us to decompose \eqref{eqn:relaxed problem} into small-scale subproblems via dual decomposition. Let $\lambda(t) \geq 0$ be the Lagrange multiplier associated with constraint \eqref{eqn:relaxed problem constraint}. The partial Lagrangian reads as follows:
\begin{multline}
    \m{L}(\b{p}, \b{N}^v, \lambda) = \int\limits_{t\in\m{T}} \bigg [ \sum_{i\in\m{K}} \sum_{j\in\m{K}} m_{ij}(t) p_{i}(t) \tau_{ij} + \lambda(t)N \\ - \big(\gamma + \lambda(t)\big) \sum_{i\in\m{K}} \big( N^v_i(t) + Q^m_{i}(t) + Q^{\t{intra}}_{i}(t) + Q^{\t{inter}}_{i}(t) \big) \bigg] \mathrm{d}t,
\end{multline}
which is clearly decomposable over zones. Thus, given dual variables $\lambda(t)$, we can optimize the following subproblem for each zone over $p_{i}(t)$ and $N^v_{i}(t)$ separately:
\begin{subequations} \label{eqn:subproblem}
\begin{align}
    \max_{\b{p}_i,\b{N}^v_i} \quad & \int\limits_{t\in\m{T}} \bigg [ \sum_{j \in \m{V}} m_{ij,t} p_{i,t} \tau_{ij} - \big(\gamma + \lambda(t)\big) N^\t{on}_i(t) \bigg ] \mathrm{d}t \label{eqn:subproblem objective} \\
    s.t. \quad & \t{relaxed system dynamics \eqref{eqn:relaxed dynamics}}, \notag \\
    & 0 \leq p_{i}(t) \leq p_\t{ub}, \ \forall t, \\
    & N^v_{i}(t) \geq N^v_{\t{lb}}, \ \forall t, \\
    & Q^w_{i}(t), Q^m_{i}(t), Q^{\t{intra}}_{i}(t), Q^{\t{inter}}_{i}(t) \geq 0, \ \forall t,
\end{align}
\end{subequations}
where $N^\t{on}_i(t) = N^v_i(t) + Q^m_{i}(t) + Q^{\t{intra}}_{i}(t) + Q^{\t{inter}}_{i}(t)$. Since each subproblem is small-scale with a two-dimensional action space and a four-dimensional state space, we can solve them in parallel using dynamic programming. However, combining solutions to all subproblems may lead to infeasible decisions that violate the coupled constraint \eqref{eqn:relaxed problem constraint}. In this case, we can always construct a feasible solution to the primal problem by the following projection:
\begin{equation} \label{eqn:projection}
    \widetilde{N}^{v}_{i}(t) = N^{v}_{i}(t) + \min\big(0, N - N^\t{on}(t) \big) \dfrac{N^v_{i}(t)}{\sum_{i \in \m{K}} N^v_{i}(t)}.
\end{equation}

\begin{algorithm}[ht!]
\SetAlgoLined
\SetKwInOut{Input}{Input}\SetKwInOut{Output}{Output}
\caption{Dual Decomposition Algorithm}
\label{alg:dual decomposition}
\Input{Current state $S(t)$}
\Output{Upper bound for the original problem}
Initialize the multipliers $\lambda(t) \geq 0$;\\
Initialize the stopping criteria;\\
\While{$\t{stopping criteria not met}$}{
Solve each \eqref{eqn:subproblem} via dynamic programming; \\
Construct a feasible solution by projection \eqref{eqn:projection}; \\ 
\eIf{$\t{Converge}$}
{Break;}
{Update the multipliers with step size $l$:
\[\lambda(t) \leftarrow \min \big(0, \lambda(t) + l \left( N^\t{on}(t) - N \right) \big)\]
}}
\end{algorithm}

In the dual decomposition algorithm, we iteratively solve the subproblem for each zone and update the multipliers $\lambda(t)$ until the stopping criterion, either zero duality gap or maximal iterations reached, is met. Details of this procedure are given in Algorithm \ref{alg:dual decomposition}. If the duality gap reduces to zero at termination, constraint \eqref{eqn:relaxed problem constraint} is satisfied and the globally optimal solution to \eqref{eqn:relaxed problem} is attained. Otherwise, constraint \eqref{eqn:relaxed problem constraint} is violated, but the derived solution is still an upper bound for \eqref{eqn:relaxed problem} according to the duality theorem. Based on Proposition \ref{prop:upper bound}, in either case, we establish a surrogate upper bound for \eqref{eqn:problem} that can be used to evaluate its solution optimality.

\begin{remark}
    In implementing the dynamic programming step, the 4-dimensional state space should be discretized, and the value function for each point in the discretized space needs to be computed. This can be tedious because there exists a large number of points in the discretized space. However, we emphasize that it will not present any computational concerns for the proposed algorithm, because the value function in the dynamic programming step can be computed in parallel, so that computation speed increases linearly with respect to the number of cores in the computer. Based on our numerical simulation on an 8-core Dell desktop, we estimate that a single instance of a commercial cloud server with 96-cores will execute Algorithm 1 within 2-3 minutes.
\end{remark}

\subsection{Implementation}
The proposed dual decomposition algorithm and the corresponding upper bound for the profit-maximization problem are theoretically applicable for an arbitrary $\m{T}$ of any length, while the computational time grows polynomially with respect to the its length. In practice, however, it is unnecessary to implement the proposed algorithm for a very large $\m{T}$, e.g., 24 hours, because the control decisions will only influence system states in the near future. For instance, it is reasonable to assume that the fleet state of the AMoD system in the afternoon does not depend on the dispatch decisions in the morning. Therefore, beyond a certain range, further extending the operation horizon will only increase the computational time without significantly improving the objective value. In addition, the proposed algorithm requires knowledge of future travel demand, which can be very accurately predicted within a short period of time, but cannot over a relatively long horizon.

In this light, we tackle the joint pricing and fleet management problem using a truncated horizon, and solve \eqref{eqn:problem} in a receding horizon manner. At each step, rather than finding a complete optimal control sequence in one go, we maximize the platform's profit within a prediction horizon $\m{T}_p$ that is shorter than $\m{T}$, e.g., 30 minutes. The proposed network flow model is embedded in the computational framework as a prediction model, enabling us to forecast future system states and optimize the control sequence within the prediction horizon. The control sequence is then partly implemented on the plant (e.g., microscopic simulation environment) and the resultant state updates are passed on to the prediction model as an initial state. At last, the prediction horizon moves forward and this procedure is repeated until the entire operation horizon ends, similar to the model predictive control (MPC) framework. In implementation of the proposed procedure, the prediction model is discretized to approximate the solution to differential equations \eqref{eqn:dynamics} for ease of computation. The temporal granularity of discretization is usually in seconds, so that the approximation is rather accurate.

\begin{remark} \label{rmk:bound}
    Since the operation horizon of the original formulation $\m{T}$ is truncated to $\m{T}_p$, the established upper bound in Section \ref{upperbound_sec} is only for the optimization with a truncated horizon, leading to an additional gap due to the truncation. However, since the dynamics of the AMoD system has a relatively short memory, this additional gap can be reduced if we properly select $\m{T}_p$, so that the consequent profit loss is marginal while at the same time the computational time is acceptable. In our simulation, $\m{T}_p$ is selected to be 30 minutes. We will validate in the numerical study that  further extending $\m{T}_p$ will be unnecessary because it only leads to negligible improvement in the platform's profit (see Fig. \ref{fig:horizon}).
\end{remark}

\section{Numerical Studies} \label{sec:numerical study}
In this section, we first validate the proposed solution framework and then numerically simulate the AMoD system to understand its operational strategies.

\subsection{Experiment Settings}
\begin{figure}[!th]
    \centering
    \includegraphics[width=0.45\linewidth]{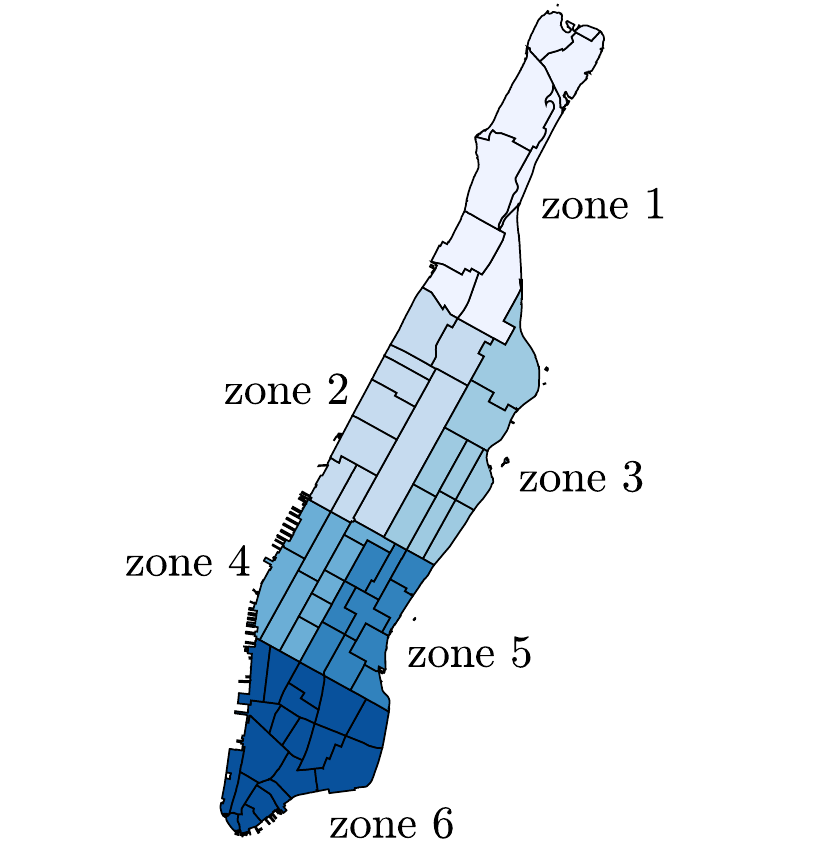}
    \caption{The six-zone partition of Manhattan, with zone 1 to 3 corresponding to Uptown and zone 4 to 6 corresponding to Midtown and Downtown.}
    \label{fig:manhattan map}
\end{figure}

\begin{figure}[!th]
    \centering
    \includegraphics[width=0.8\linewidth]{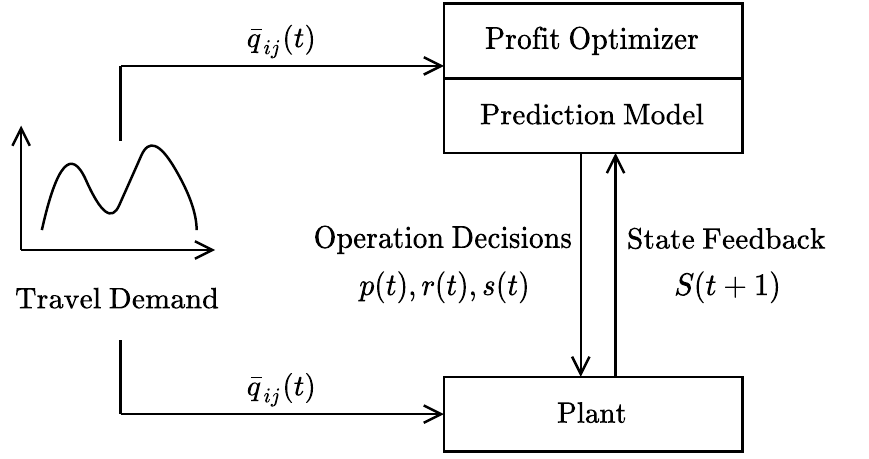}
    \caption{Schematic diagram of the MPC framework.}
    \label{fig:MPC}
\end{figure}
In the numerical studies, We use the well-known New York City Taxi and Limousine Commission dataset \cite{tlc}, which records in detail when and where each trip starts and ends. We only focus on the trips starting and ending in Manhattan. As shown in Fig. \ref{fig:manhattan map}, the island is partitioned into 6 zones.  Travel demand for each OD pair is synthesized based on the dataset. Given the OD demand, starting and ending locations of trips are randomly generated within the zone such that the average trip distance is close to real data. The total fleet size is set to be 3,000. The operational cost of each vehicle is assumed to be $\gamma=\t{\$10/hr}$. The lower bound of idle vehicle in each zone is set as 15 and the price ceiling is set as \$2.5/min. The length of operation horizon is 24 hours, while the length of prediction horizon in the MPC implementation is set to be 30 minutes. Fig. \ref{fig:MPC} is a schematic diagram of the computational framework, where the macroscopic network flow model is utilized as the prediction model and a microscopic simulation testbed is built as the plant. The synthetic demand data are sequentially fed into the profit optimizer and the plant. Other parameters in the prediction model are calibrated to match real data and the performance of the plant as next introduced. Simulations are carried out in MATLAB using a Dell desktop with Intel Core i7-9700 CPU and 32 GB RAM.

\subsection{Microscopic Simulation Testbed}
In the microscopic simulation environment, we track each vehicle's status, location, and destination, and also the information of passengers that are waiting to be matched. Idle cars will cruise in a zone until it is assigned a trip (either delivery or relocation) or sent offline. All vehicles except those off-duty ones move along the line segment connecting their origins and destinations at a constant speed. At each step, potential passengers make entry decisions depending on the price and pickup time. Greedy dispatch search method in \cite{nourinejad_ride-sourcing_2020} is adopted and thus the pickup time is computed as the travel time between the pickup location and the closest available vehicle. Passengers enter the system with heterogeneous trip information and will be matched right away if an idle vehicle is available in the same zone. Otherwise, they will wait for available cars. Each passenger has a maximum tolerance queuing time beyond which the passenger will cancel the request and exit from the system. The discrete-time system evolves with a step size of 20 seconds.

\begin{figure*}[!t]
\centering
\subfloat[Convergence at 8:00]{\includegraphics[width=0.35\textwidth]{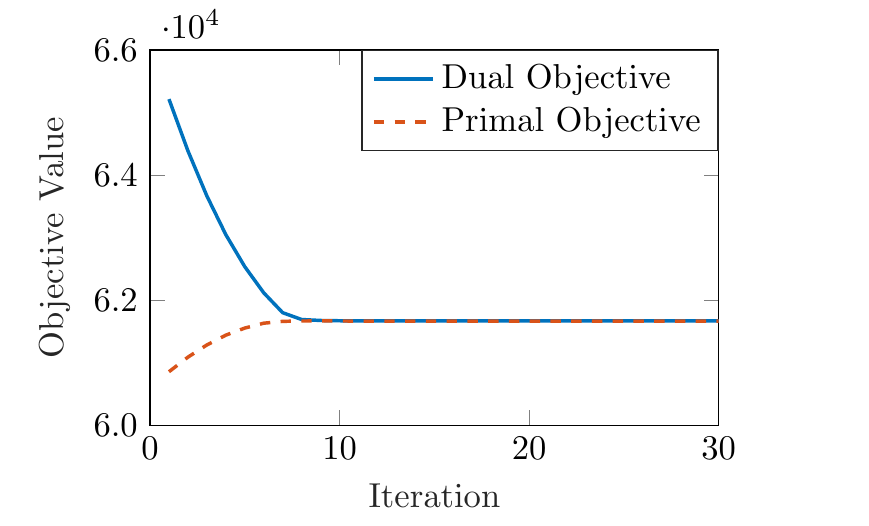}%
\label{fig:convergence 0800}}
\hspace{-0.8cm}
\subfloat[Convergence at 16:00]{\includegraphics[width=0.35\textwidth]{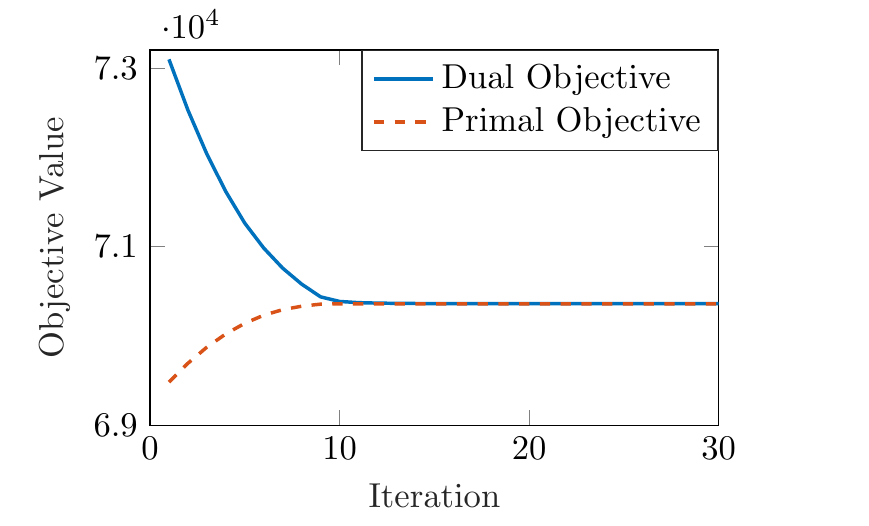}%
\label{fig:convergence 1600}}
\hspace{-0.8cm}
\subfloat[Convergence at 20:00]{\includegraphics[width=0.35\textwidth]{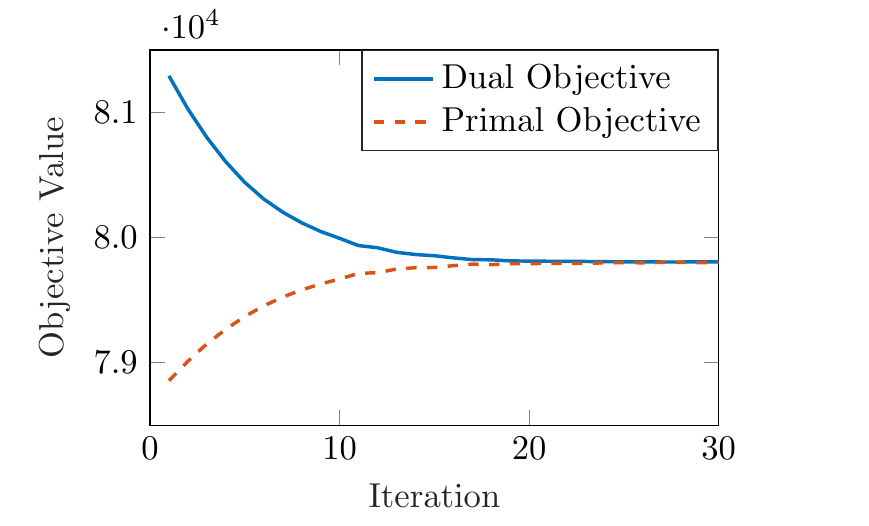}%
\label{fig:convergence 2000}}
\caption{Convergence performance of Algorithm \ref{alg:dual decomposition} at at different times of a day.}
\label{fig:convergence}
\end{figure*}

\begin{figure*}[!t]
\centering
\subfloat[Profit Comparison]{\includegraphics[width=0.35\textwidth]{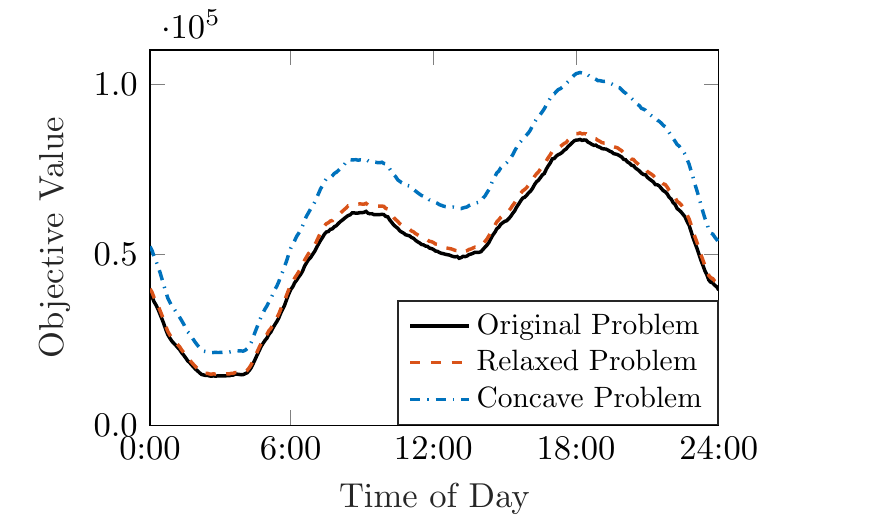}%
\label{fig:bound comparison}}
\hspace{-0.8cm}
\subfloat[Optimality Gap]{\includegraphics[width=0.35\textwidth]{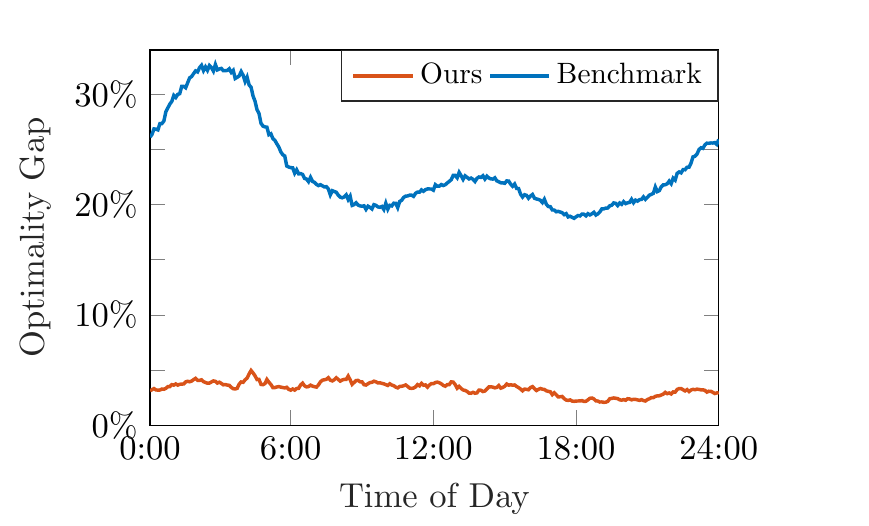}%
\label{fig:bound gap}}
\hspace{-0.8cm}
\subfloat[Profit versus Horizon Length]{\includegraphics[width=0.35\textwidth]{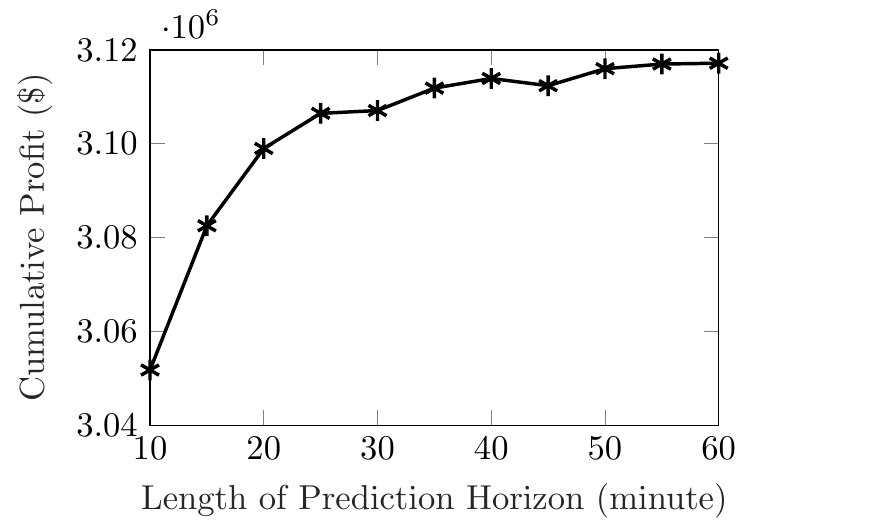}%
\label{fig:horizon}}
\caption{(a) compares the solution to \eqref{eqn:problem}, the established upper bound, and the benchmark upper bound based on a concave relaxation; (b) quantifies the optimality gap of the two bounds; (c) evaluates how the length of prediction horizon will affect profitability.}
\end{figure*}

\subsection{Performance of Upper Bounds}
We first solve the original problem \eqref{eqn:problem} is solved using interior-point methods, and then derive an upper bound by solving the relaxed problem. In Algorithm \ref{alg:dual decomposition}, all multipliers are initialized to 0.05, and the step size is set to 1e-5 with a decaying factor 0.99. Fig. \ref{fig:convergence} showcases the convergence performance at 8:00, 16:00, and 20:00, respectively. Our results show that the duality gap gradually converges to zero, and the performance remains consistent throughout the day. This confirms that we have obtained the {\em globally} optimal solution to the relaxed problem \eqref{eqn:relaxed problem}, and thus established a theoretical upper bound for the original problem \eqref{eqn:problem}.

To evaluate the quality of this bound, we further benchmark it against another theoretical upper bound built upon a simplified model \cite{chen_real-time_2023}, where demand elasticity with respect to waiting time,   vehicle-passenger matching friction, and request confirmation are not considered.  The idea is that the profit-maximization problem can be relaxed to a concave problem with linear dynamics and a concave objective function, where the globally optimal solution can be easily obtained to establish an upper bound for the original problem. Note that the modeling settings between our work and the benchmark case \cite{chen_real-time_2023} are distinct (e.g., we consider demand elasticity with respect to waiting time), thus the bound derived in \cite{chen_real-time_2023} is not directly applicable in our context. However, we can still compare these two by using the same idea in \cite{chen_real-time_2023} to relax our original problem to a concave one by simplifying our model. The detailed formulation and derivation can be found in the Appendix. Fig. \ref{fig:bound comparison} presents objective values of the original problem \eqref{eqn:problem}, the relaxed problem \eqref{eqn:relaxed problem}, and the concave problem \eqref{eqn:concave problem}, respectively, and Fig. \ref{fig:bound gap} provides a comparison between the two upper bounds. The results suggest that the proposed relaxation and decomposition method produces a better upper bound than the concave relaxation, with a smaller optimality gap of less than 5\% versus more than 20\%. This demonstrates that our method can provide a nontrivial and accurate reference for assessing the solution optimality of the original problem, despite its nonconvexity.

\begin{remark}
In Fig. \ref{fig:bound gap}, the optimality gap is relatively higher late at night (5\% at around 3:00), due to an underestimation of the duration of inter-zone trips in the proposed relaxation scheme. The inaccuracy is magnified because long-distance inter-zone trips account for a significant proportion of the total demand at late night. Fortunately, demand is at its lowest level during this period, and thus the temporarily higher gap will not lead to a significant profit loss.
\end{remark}

Besides, the length of prediction horizon is a crucial hyperparameter in the proposed computational framework. To examine its impacts on solution performance, we vary the prediction horizon from 10 to 60 minutes and compare the platform's cumulative profit. As shown in Fig. \ref{fig:horizon}, the results align with our intuition that a longer prediction horizon leads to increased profit. However, the improvement becomes marginal as the prediction horizon further extends. The curve in Fig. \ref{fig:horizon} follows an upward trend but is not consistently monotonic because of the randomness in the microscopic simulation environment. Considering a long prediction horizon will result in a large-scale optimization and thus impede real-time implementation, we fix it to be 30 minutes in subsequent case studies to balance the optimization performance and computation complexity. These results also echo our discussion in Remark \ref{rmk:bound} that the optimality gap due to the gap between $\m{T}$ and $\m{T}_p$ is insignificant at the selected length.

\subsection{Joint Pricing and Rebalancing}
\begin{table*}[!ht]
\centering
\caption{Comparison between Joint Pricing \& Rebalancing and Pricing Only.\label{tab:comparison}}
\begin{tabular}{cccccccccccc}
\hline
\multirow{2}{*}{} & \multicolumn{3}{c}{24-hr normal period} &  & \multicolumn{3}{c}{30-min demand surge in zone 3} &  & \multicolumn{3}{c}{30-min demand surge in zone 4} \\ \cline{2-4} \cline{6-8} \cline{10-12} 
& with rebalance & w/o rebalance & \% gap &  
& with rebalance & w/o rebalance & \% gap &  
& with rebalance & w/o rebalance & \% gap \\ \hline
profit    & 2.5437e+06      & 2.5126e+06   & 1.09\% &  & 8.6274e+04          & 8.1065e+04       & 7.62\%     &  & 8.7910e+04          & 7.9004e+04       & 11.27\%    \\
trips     & 1.3442e+05      & 1.3306e+05   & 1.02\% &  & 3.8510e+03          & 3.7450e+03       & 2.83\%     &  & 4.0170e+03          & 3.9030e+03       & 2.92\%     \\ \hline
\end{tabular}
\end{table*}

\begin{figure*}[!th]
\centering
\subfloat[With Rebalance. Demand Surges in Zone 4.]{\includegraphics[width=1.0\textwidth]{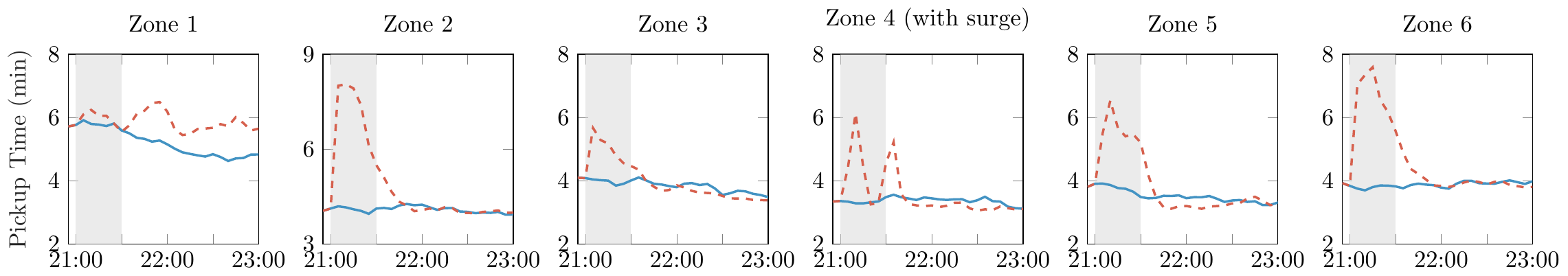}%
\label{fig:surge 2100 4}}
\vspace{0cm}
\subfloat[Without Rebalance. Demand Surges in Zone 4.]{\includegraphics[width=1.0\textwidth]{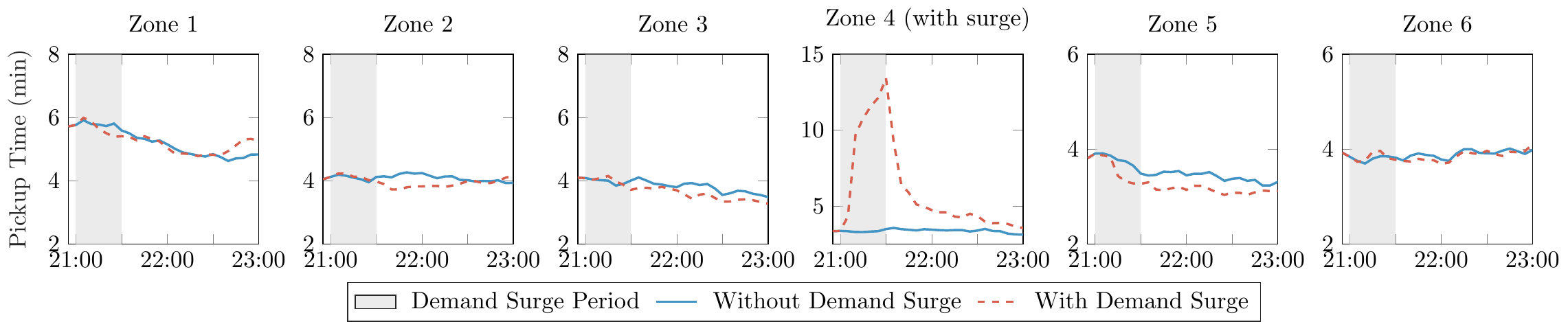}%
\label{fig:surge 2100 4 pricingonly}}
\caption{Comparing passengers' pickup time with and without fleet rebalance when a demand shock (highlighted in gray) is imposed on different zones.}
\label{fig:surge}
\end{figure*}

In the proposed model, the platform jointly makes pricing and fleet rebalancing decisions, whilst its value remains unclear. To investigate the impacts of joint operation, we compare the cumulative profit and total trips delivered with the case of pricing only, where fleet rebalancing is disabled. The results are summarized in Tab. \ref{tab:comparison}. Surprisingly, fleet rebalancing does not lead to notable benefits, with merely a 1\% increase in profit and throughput. When future demand is accurately predicted, the platform can plan vehicle dispatch long ahead of time to minimize the total rebalancing trips. However, differentiated pricing in space and time can also coordinate demand and supply, although with less flexibility than fleet rebalancing. Therefore, joint pricing and rebalancing will only lead to minor improvements during the normal period.

In contrast, the benefits of fleet rebalancing stand out when demand prediction is seriously off, e.g., demand surges abruptly due to the interruption of public transit. In this case, the spatial disparity between supply and demand is significantly escalated, necessitating prompt and flexible fleet dispatch that pricing only cannot efficiently provide. To emulate the unexpected demand shock, we triple the travel demand in one of those six zones for 30 minutes. The spike is respectively imposed on zone 3 and 4 at 21:00, when no spare vehicles are available to help service the surging needs. During the 30-minute demand surge period, joint supply-demand coordination can increase the platform's profit by up to 10\% and its throughput by around 3\%, which highlights the importance of fleet management under system perturbations.

We further show passengers' pickup times with and without fleet rebalancing in Fig. \ref{fig:surge}. These results reveal the fundamental difference between the two operational strategies during the surge period. With fleet rebalancing, the platform proactively dispatches vehicles from neighboring areas to support the zone short of supply. Pickup time in each zone thus increases at varying degrees, indicating that impacts of demand shock propagate in the city. However, when fleet rebalancing is not considered, the platform responds to demand surge only by raising prices in a single zone. Thereby we see in Fig. \ref{fig:surge 2100 4 pricingonly} that passengers' pickup time drastically increases in zone 4 (from 3 to 13 minutes) while all other zones are almost unaffected. Moreover, the quality of service returns to normal in about 30 minutes with fleet rebalancing, but in over 1 hour without it, showing that impacts dissipate faster when the fleet can be rebalanced. This sheds light on the fact that joint pricing and rebalancing can improve system responsiveness to unexpected scenarios at the cost of neighboring zones.

\subsection{Analysis of operational Strategies}
\begin{figure*}[!t]
\centering
\subfloat[Average Trip Fare]{\includegraphics[width=0.35\textwidth]{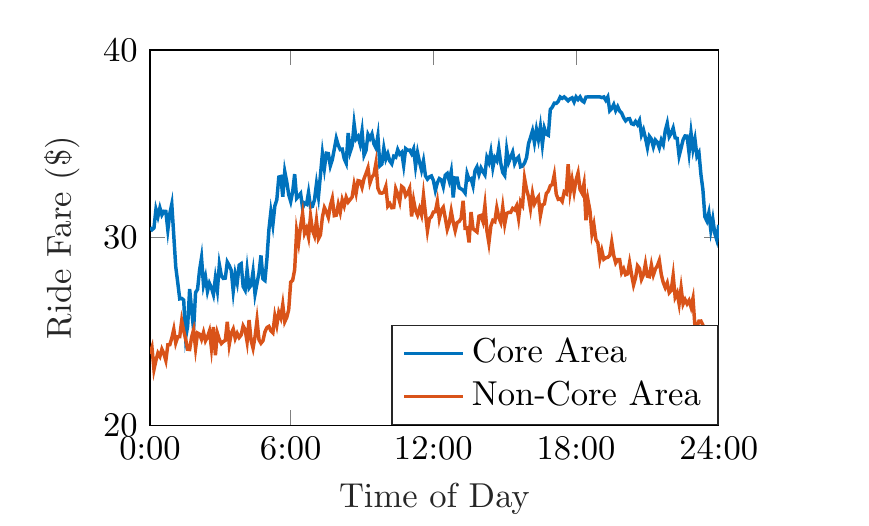}%
\label{fig:trip fare}}
\hspace{-0.8cm}
\subfloat[Rebalancing Flow]{\includegraphics[width=0.35\textwidth]{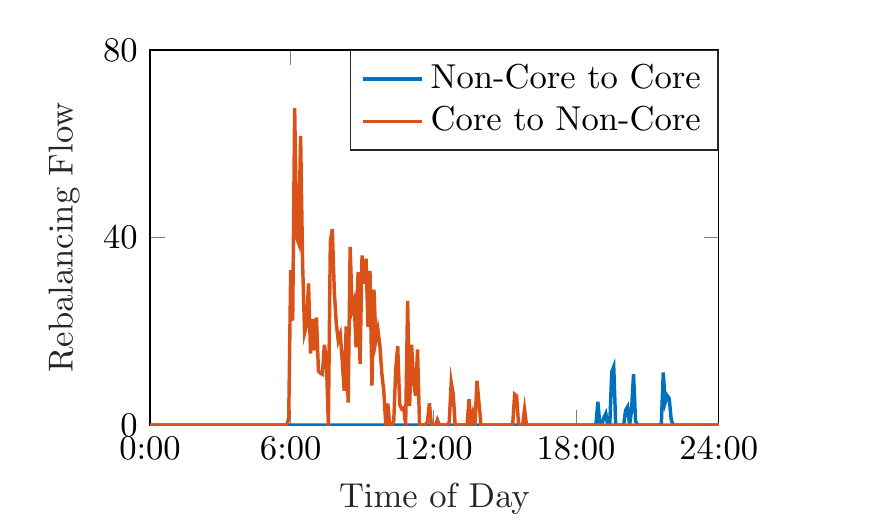}%
\label{fig:rebalance flow}}
\hspace{-0.8cm}
\subfloat[Parking Distribution]{\includegraphics[width=0.35\textwidth]{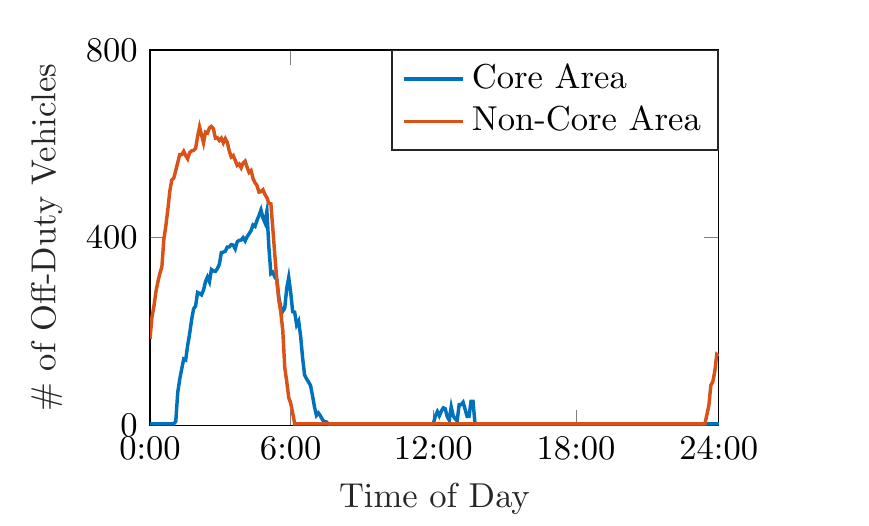}%
\label{fig:parking distribution}}
\caption{The platform's operational strategies in a day, with zone 1-3 being aggregated as the non-core area and zone 4-6 being aggregated as the core area.}
\label{fig:strategy}
\end{figure*}

To understand the platform's operational strategies in both spatial and temporal dimensions, we simulate the system evolution for 24 hours. According to real data, over 68\% of trips in Manhattan originate from zones 4, 5, and 6. Therefore, we categorize the island into two distinct areas based on their travel demand levels to facilitate analysis, with zones 1 to 3 defined as the non-core area and zones 4 to 6 defined as the core area. Fig. \ref{fig:trip fare} presents the fare for a 15-minute trip starting from different areas, Fig. \ref{fig:rebalance flow} presents the number of vehicles being relocated between the two areas, and Fig. \ref{fig:parking distribution} showcases the spatial distribution of off-duty vehicles.

It is clearly observed that the spatial imbalance of demand has a significant impact on the platform's operational strategies. In the morning, many passengers flood in the urban core, leading to insufficient supply in the non-core area. Whereas commuters return to the non-core area in the evening, leaving the urban core short of vehicles. The asymmetric demand results in an unbalanced fleet distribution, which requires external interventions to mitigate the spatial mismatch between supply and demand and maintain the quality of service. In response to the upcoming evening rush hours, for instance, the platform lowers the fare in the non-core area at around 17:30. On the contrary, prices in urban cores remain high until 19:00. Such a pricing policy can suppress demand in high-demand zones and meanwhile incentivize trips that depart from low-demand zones. While during morning rush hours, as shown in Fig. \ref{fig:rebalance flow}, lots of vehicles are dispatched to non-core area to serve the demand peak in the morning.

Aside from pricing and rebalancing, the platform also resorts to fleet size adjustments to steer the spatial distribution of supply. According to Fig. \ref{fig:parking distribution}, the fleet is proactively resized in accordance with the time-varying demand levels. For example, a part of the fleet will be parked in the early afternoon (from 12:00 to 13:30) and late at night (from 23:00 to 6:00). These periods align with the off-peak hours, during which vehicles are oversupplied. The platform thus deactivates redundant cars to circumvent high operational costs. The temporal pattern of parking behaviors echoes the finding in \cite{xu_optimal_2017} that demand for parking spaces is inversely proportional to passenger demand during off-peak hours. Nonetheless, the two parking periods demonstrate distinct spatial patterns. More off-duty vehicles are retained in the non-core area at night but in the core area at noontime. This again results from the demand imbalance. After the evening peak, cars concentrate in the non-core area. The platform tends to leave them there for overnight parking, which can avoid the operational cost of relocation trips and also prepare for the next morning peak. Likewise, vehicles accumulate in the urban core after the morning peak and then are parked right there at noontime off-peak hours so as to serve the forthcoming evening peak. These findings reveal that the provision of parking spaces is intimately related to the spatiotemporal demand pattern, which may advise infrastructure planning.

\section{Conclusion} \label{sec:conclusion}
This paper considers the spatiotemporal pricing and fleet management strategies for the AMoD platform taking an elastic demand model into account. The system dynamics are formulated as a network flow model, which simultaneously captures essential components, including passenger waiting time, demand elasticity, vehicle-passenger matching, etc. Model predictive control is used to compute a near-optimal solution, and a relaxation scheme is proposed to transform the original problem into a decomposable one. The relaxed problem is then solved by dual decomposition and dynamic programming, which establishes a theoretical upper bound to evaluate the solution optimality. Compared to the benchmark upper bound based on a concave relaxation, our bound provides a more accurate estimation of the globally optimal solution to the original nonconvex problem.

Comprehensive case studies for Manhattan are conducted to understand the operations of AMoD systems. First, we find that when demand can be accurately predicted, jointly making pricing and rebalancing decisions only leads to minor profit and throughput improvements than the case of pricing only, because a pricing strategy with spatial and temporal differentiation can also coordinate demand and supply. However, joint pricing and rebalancing can improve profits by up to 10\% under abrupt demand surges. We also show that with fleet rebalancing, the impacts of demand shocks will propagate broader but dissipate faster. Furthermore, numerical results suggest that the platform’s operation decisions are made in response to the asymmetric demand patterns. Specifically, vehicles are dispatched from core to non-core areas at the morning peak and are relocated along the opposite direction at the evening peak. Besides, most redundant vehicles are parked in non-core areas during evening off-peak hours but in urban cores during noontime off-peak hours, which provides insights for the provision of parking spaces. 

Future studies can extend the proposed model into a stochastic setting with uncertain demand. Another promising extension is to take ride-pooling services into account, possibly by introducing a discount ratio in the aggregate matching function \eqref{eqn:matching function} as \cite{xu_generalized_2021} did. In addition, one may incorporate congestion externality into the model to examine how AMoD systems will affect urban transportation systems. The charging strategy in electric AMoD systems also warrants further investigation.

\appendix
\section*{Validation of Request Cancellation Model}
The performance of the request cancellation model \eqref{eqn:cancallation} is validated using empirical data generated in the microscopic simulation environment. We assume passengers' maximum tolerance time is normally distributed with mean 2 minutes and standard deviation 30 seconds. As shown in Fig.~\ref{fig:request cancellation evaluation}, requests will be canceled only when idle cars are inadequate while waiting passengers accumulate. The empirical data are fitted by the black plane in grids and it demonstrates that a linear model can well characterize the relationship among cancellations, idle cars, and waiting passengers.
\begin{figure}[!th]
\centering
\includegraphics[width=0.95\linewidth]{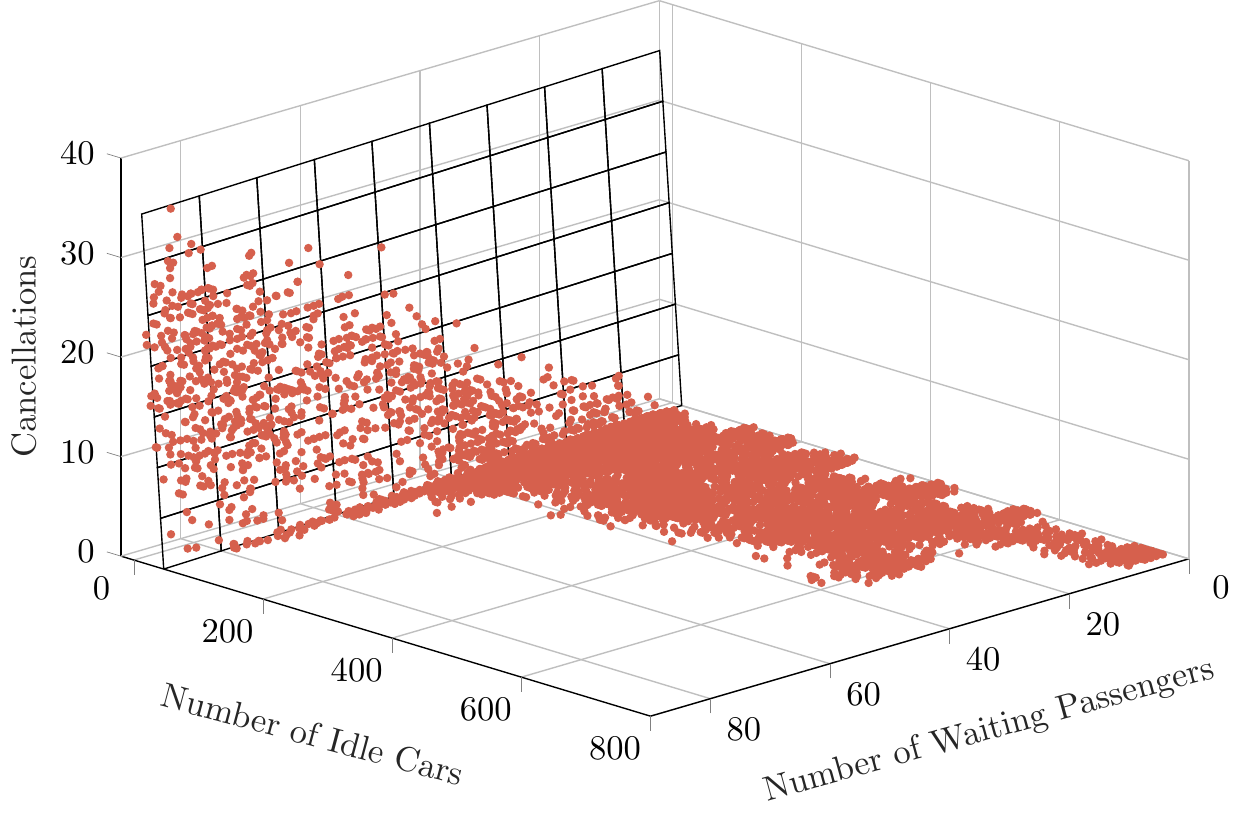}
\caption{Validation of the linear request cancellation model.}
\label{fig:request cancellation evaluation}
\end{figure}

\section*{Validation of Trip Completion Model}
The linear trip completion model \eqref{eqn:trip completion} is assessed using real data from \cite{tlc}. Fig.~\ref{fig:trip completion evaluation} presents the number of on-board passengers ($Q^b_{ij}$) versus the number of completed trips ($\tilde{q}_{ij}$) for different OD pairs. It can be observed that $Q^b_{ij}$ and $\tilde{q}_{ij}$ demonstrate a strong linear correlation. Therefore, the linear model \eqref{eqn:trip completion} suffices to capture the dynamics of trip completion.
\begin{figure}[!th]
\centering
\includegraphics[width=0.95\linewidth]{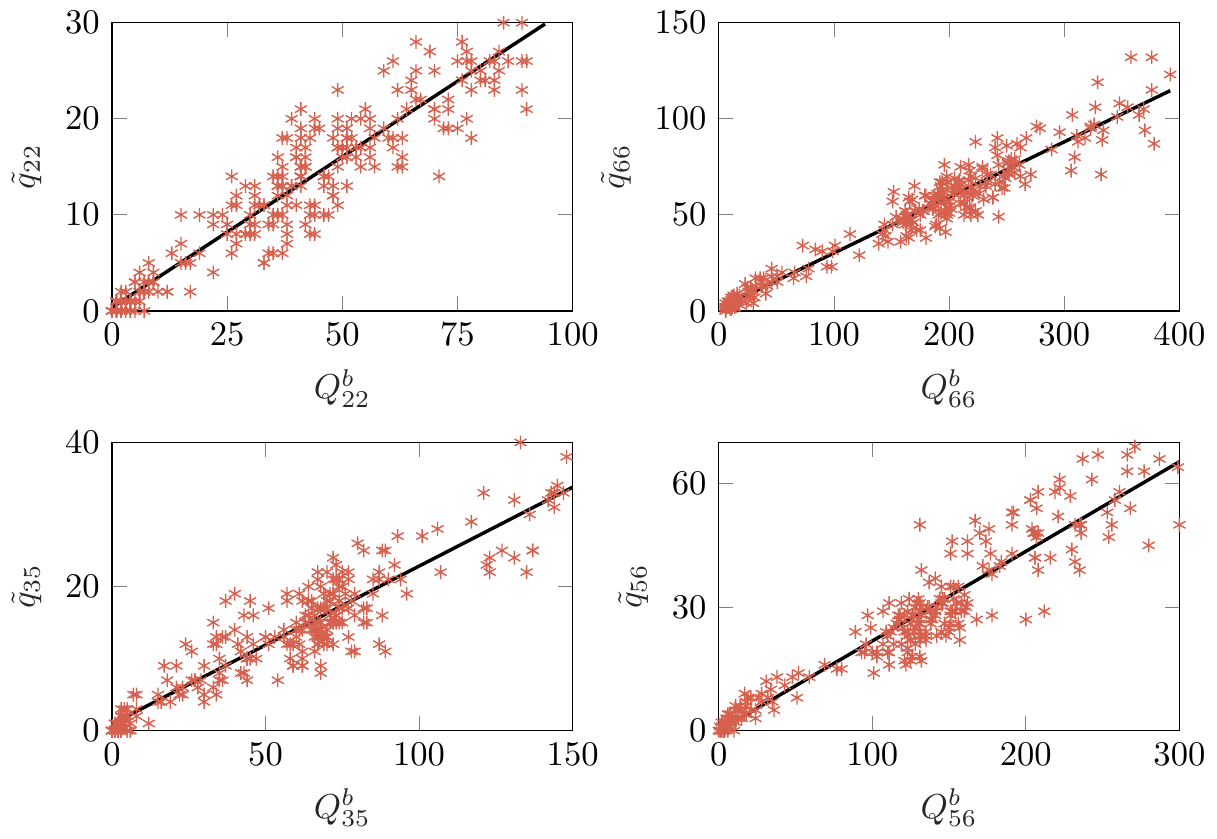}
\caption{Validation of the linear trip completion model for different OD pairs.}
\label{fig:trip completion evaluation}
\end{figure}

\section*{Concave Relaxation and Benchmark}
To offer a benchmark for the established upper bound to compare against, we adopt a concave formulation of the spatiotemporal pricing and fleet management problem as in \cite{chen_real-time_2023}, where request confirmation and cancellation as well as passenger pickup are not considered.

Specifically, to make this concave relaxation work in our context, two major modifications are made compared to the original formulation \eqref{eqn:problem}. First, instant request confirmation and passenger-vehicle matching are assumed. Each new-coming passenger is matched to and picked up by an idle vehicle immediately after placing the request. As such, system states are reduced to $Q^b_{ij}(t)$, $N^v_{i}(t)$, $N^r_{ij}(t)$, and $N^p_{i}(t)$, and system dynamics become linear, which significantly enhances the problem's tractability. To avoid a trivial bound, passenger waiting time $w_i(t)$ is fixed at its smallest possible value\footnote{We choose this constant value by reasonably assuming that there exists an upper bound on the proportion of vehicles concentrated in a single zone (e.g., 1/3). Intuitively, this upper bound avoids the extreme case where a significant amount of idle vehicles are concentrated in one zone, while all other zones are underserved.} instead of 0. Second, considering the demand function \eqref{eqn:demand function} is invertible, demand rate $z_{ij}(t) = q_{ij}(t)/\bar{q}_{ij}(t) \in [0, 1]$ is used as a decision in place of price $p_i(t)$. The platform's revenue then becomes a function of $z_{ij}(t)$. Although the revenue may not be concave in price, once the waiting time is fixed, the revenue is known to be concave in demand rate for most demand functions, such as linear, exponential, logit, among others \cite{chen_real-time_2023}. Note that $z_{ij}(t)$ is OD dependent and so is the corresponding trip fare. This relaxation is to preserve concavity after the change of variables.

With these two modifications, the system becomes linear:
\begin{subequations} \label{eqn:linear dynamics}
\begin{align}
    & \dfrac{\mathrm{d} Q^b_{ij}(t)}{\mathrm{d} t} = q_{ij}(t) - \tilde{q}_{ij}(t), \vphantom{\sum_{j\in\m{K}}} \label{eqn:linear dynamics Qb}\\
    & \dfrac{\mathrm{d} N^v_{i}(t)}{\mathrm{d} t} = s_{i}(t) + \sum_{j\in\m{K}} \tilde{q}_{ji}(t) + \sum_{j\in\m{K}} \tilde{r}_{ji}(t) \label{eqn:linear dynamics Nv} \\
    & \phantom{\dfrac{\mathrm{d} N^v_{i}(t)}{\mathrm{d} t} = \quad \quad \quad \quad} - \sum_{j\in\m{K}} q_{ij}(t) - \sum_{j\in\m{K}} r_{ij}(t), \notag \\
    & \dfrac{\mathrm{d} N^r_{ij}(t)}{\mathrm{d} t} = r_{ij}(t) - \tilde{r}_{ij}(t), \vphantom{\sum_{j\in\m{K}}} \label{eqn:linear dynamics Nr}\\
    & \dfrac{\mathrm{d} N^p_{i}(t)}{\mathrm{d} t} = -s_{i}(t). \vphantom{\sum_{j\in\m{K}}} \label{eqn:linear dynamics Np}     
\end{align}
\end{subequations}
And the profit maximization problem becomes concave:
\begin{subequations} \label{eqn:concave problem}
\begin{align}
    \max_{\b{z},\b{r},\b{s}} \ & \int\limits_{t\in\m{T}} \Bigg [ \sum_{i\in\m{K}} \sum_{j\in\m{K}} q_{ij}(t) p_{i}(t) \tau_{ij} - \gamma N^\t{on}(t) \Bigg] \mathrm{d}t \label{eqn:concave problem objective} \\
    s.t. \quad & \t{linear system dynamics \eqref{eqn:linear dynamics}},  \notag \\
    & r_{ij}(t) \geq 0, \ \forall i,j,t, \label{eqn:concave r bound} \\
    & 0 \leq z_{i}(t) \leq 1, \ \forall i,t, \label{eqn:concave q bound} \\
    & N^v_{i}(t) \geq N^v_{\t{lb}}, \ \forall i,t, \label{eqn:concave Nv bound}\\
    & 0 \leq N^p_{i}(t) \leq N^p_{\t{ub},i}(t), \ \forall i,t, \label{eqn:concave Np bound}\\
    & Q^b_{ij}(t), N^b_{ij}(t) \geq 0, \ \forall i,j,t. \label{eqn:concave state bound}
\end{align}
\end{subequations}
Since the relaxed profit-maximization problem \eqref{eqn:concave problem}  is concave, its globally optimal solution can be easily derived, which provides  an upper bound for the original problem \eqref{eqn:problem}. This bound is considered as the benchmark and is compared against the proposed decomposition and dynamic programming approach.

\bibliographystyle{IEEEtran}
\bibliography{TITS2022}

\end{document}